\documentclass{article}
\usepackage[utf8]{inputenc}

\usepackage{amsmath}
\usepackage{amsfonts}
\usepackage{amsthm}
\usepackage{mathtools}
\usepackage{biblatex}

\usepackage[english]{babel}
\newtheorem{theorem}{Theorem}[section]
\newtheorem{corollary}{Corollary}[theorem]
\newtheorem{lemma}[theorem]{Lemma}

\title{Bioperational Multisets in Various Semi-rings}
\author{O.M. Cain \\ onnomawcain@gmail.com}
\date{July 2019}

\begin{document}

\maketitle
\textbf{Abstract.} One can find lists of whole numbers having equal sum and product. We call such a creature a \textit{bioperational multiset}. No one seems to have seriously studied them in areas outside whole numbers such as the rationals, Gaussian integers, or semi-rings. We enumerate all possible sum-products for a bioperational multiset over whole numbers and six additional domains.

\section{Introduction}
The numbers $1,2,$ and $3$ have the property that their sum is also their product. That is, $1+2+3=1\cdot 2\cdot 3=6$. As Matt Parker [1] has pointed out, this gives us as a strange sort of byproduct:
$$\log (1+2+3) = \log 1 + \log 2 + \log 3,$$
due to the identity
$$\log (ab)=\log a+\log b.$$

We coin the term \textit{bioperational} here to refer to any such multiset $\{a_i\}_{i=1}^n$ having an equal sum and product. That is to say,
$$\sum_{i=1}^n a_i = \prod_{i=1}^n a_i.$$
A \textit{multiset}, as can be guessed, is a set in which a number can occur multiple times (instead of just once or not at all).

There are some open conjectures about the number of bioperational multisets over $\mathbb{N}$ of length $n$ as $n$ gets bigger and bigger [2]. There is also a smattering of analysis available on math.stackexchange including uniqueness of solutions and connections with trigonometry [3][4][5][6]. One can also find a surprisingly complicated solution algorithm [7]. But it seems no one has formally categorized bioperational multisets over $\mathbb{N}$ (unless we count Matt's ``pseudoproof" and the passing comments of others).

In addition to $\mathbb{N}$, it is interesting to explore bioperational multisets in other environments. They are well-defined anywhere addition and multiplication are well-defined (hence in any semi-ring). We would suspect therefore to find these creatures lurking in $\mathbb{Z}, \mathbb{C}, \mathbb{F}_p, \text{GL}_n(\mathbb{R})$, and many other places (even in non-Abelian rings!). In this paper, we limit ourselves to analyzing bioperational multisets over
\begin{itemize}
    \item non-negative integers ($\mathbb{N}$) in Section 3,
    \item integers ($\mathbb{Z}$) in Section 4,
    \item general fields ($\mathbb{Q}, \mathbb{F}_p,$ etc) in Section 5,
    \item lunar integers ($\mathbb{L}$) in Section 6,
    \item Gaussian integers ($\mathbb{Z}[i]$) in Section 7,
    \item Eisenstein integers ($\mathbb{Z}[\omega]$) in Section 8,
    \item and integers with $\sqrt{2}$ appended ($\mathbb{Z}[\sqrt{2}]$) in Section 9.
\end{itemize}

\section{Some Definitions}
Firstly, we have to blow some dead leaves out of the way to see clearly. To do so requires the leafblower of \textit{vocabulary}. Suppose we have a bioperational multiset $S=\{a_i\}_{i=1}^n$. We say  
\begin{itemize}
    \item $S$ is \textit{trivial} if it contains only one element, 
    \item $S$ \textit{vanishes} if the sum-product is zero, and
    \item $S$ is \textit{minimal} if no proper subset of its terms forms a bioperational set of the same sum-product.
\end{itemize}
Note that we are using `trivial' differently than in [2].

All the examples we are considering are 1) integral domains and 2) Abelian. That means 1) if any two numbers have a product of zero then one or both of them must also be zero ($ab=0$ implies $a=0$ or $b=0$) And 2) the order we multiply stuff doesn't matter (so $ab=ba$). This has some consequences on our analysis.

1) In an integral domain, any vanishing bioperational multisets must contain zero -- which is rather boring. An analysis of vanishing bioperational multisets over non-integral domains might be interesting. In fact, the first example $(1,2,3)$ vanishes if we place it in $\mathbb{Z}/6\mathbb{Z}$. But the adventure of non-integral domains in general will be neglected here. 

2) Since the order of multiplication doesn't matter in our examples, we call our subjects bioperational \textit{multisets}. However, the author greatly hopes that bioperational multisets will be explored in non-Abelian rings (in which case they would be bioperational \textit{sequences}). The author would have loved to explore these objects in the quaternions ($\mathbb{H}$) themselves but was too ignorant for the attempt.

For two multisets $A$ and $B$ we let ``$A+B$" denote their  \textit{multiset sum} which is best explained with an example:
$$\{2,7,2,2,3\}+\{1,2,7,7\}=\{1,2,2,2,2,3,7,7,7\}.$$
In technical terms, we are summing the multiplicities of all elements involved. Similarly ``$A-B$" will denote subtracting multiplicites:
$$\{2,7,2,2,3\}-\{2,2,3\}=\{2,7\}.$$
We also use coefficients of multisets to denote scaling multiplicities. Or in other words
$$3\{2,5,5\}=\{2,2,2,5,5,5,5,5,5\}.$$

Lastly, to keep our equations less messy, for a multiset $S$ we will write $\sigma(S)$ for the sum of its elements and $\pi(S)$ for the product. They look nicer than $\sum_{i=1}^n a_i$ and $\prod_{i=1}^n a_i.$

\section{Non-negative integers ($\mathbb{N}$)}
We begin with 
\begin{theorem}
    There is exactly one non-vanishing bioperational multiset over $\mathbb{N}$ of length $n$ for $n=2,3,4$ with constructions
    $$2+2=2\cdot2=4,$$
    $$1+2+3=1\cdot2\cdot3=6,$$
    $$\text{and}\quad 1+1+2+4=1\cdot1\cdot2\cdot4=8.$$
    This confirms Matt's conjecture for $n=2$ and is stated without proof in [2].
\end{theorem}
\begin{proof} 
    We first take $n=2$. Suppose we have $ab=a+b$. Rearrangement yields $ab-a-b+1=(a-1)(b-1)=1$. Since the only way to factor $1$ as two positive integers is $1=1\cdot 1$ it follows that $a-1=b-1=1$ or equivalently, that $a=b=2$.
    
    A phenomenal proof of $n=3$ was given by Mark Bennet in [3]. We repeat it here. Suppose we have $a+b+c=abc$. At least one term must be $1$ since if otherwise $a\ge b\ge c\ge 2$ and we would have
    $$3a\ge a+b+c=abc\ge 4a$$
    which is true only if $a\le 0$. But that is a contractiction since we are assuming $a\ge 2$.. 
    
    Okay, so let $c=1$. Then we have $a+b+1=ab$. Rearranging yields $ab-a-b+1=(a-1)(b-1)=2$. It follows that $\{a-1,b-1\}=\{1,2\}$ or equivalently, that $\{a,b\}=\{2,3\}$.
    
    Similarly for $n=4$, suppose $a+b+c+d=abcd$. Similar to the case $n=3$, we know at least one term must be $1$ since $a\ge b \ge c\ge d\ge 2$ would imply
    $$4a\ge a+b+c=abc\ge 8a.$$
    So let $d=1$. But we can play the same trick again. If we have $a\ge b\ge c \ge 2$ then
    $$3a+1\ge a+b+c+1=abc \ge 4a$$
    from which it follows that $a\le 1$ which again contradicts our assumption $a\ge 2$. Let $c=1$ also.
    
    We have $a+b+2=ab$. Rearranging yields $(a-1)(b-1)=3$ from which it follows $\{a,b\}=\{4,2\}$.
\end{proof}

One may be tempted to generalize this proof technique and keep tackling larger and larger $n$ (In fact, we wrote a program to do exactly this [7]. See [8] for another such solution algorithm). For example, $n=5$ yields
\begin{theorem}
    There are $3$ non-vanishing bioperational multisets over $\mathbb{N}$ of length $n=5$ with constructions
    $$1+1+2+2+2=1\cdot1\cdot2\cdot2\cdot2=8,$$
    $$1+1+1+3+3=1\cdot1\cdot1\cdot3\cdot3=9,$$
    $$\text{and}\quad 1+1+1+2+5=1\cdot1\cdot1\cdot2\cdot5=10.$$
\end{theorem}
\begin{proof}
    From computation.
\end{proof}

But there turns out to be an easier way to catalog all bioperational multisets. We first need a lower foothold (or ``lemma" as they're called).
\begin{lemma}
    The product of one or more real numbers, all greater than or equal to $2$, is greater than or equal to their sum. That is, if $a_i\ge 2$, $i=1,..,n$ then
    $$\prod_{i=1}^n a_i \ge \sum_{i=1}^n a_i.$$
\end{lemma}
\begin{proof}
    Induction will be used on $n$. The base case, $n=1$, of a single number is clearly true since every number is equal to itself -- and therefore is greater than or equal to itself ($a_1\ge a_1$).
    
    Next suppose we have some multiset $S=\{a_i\}_{i=1}^n$ for which the theorem statement is true. So $\pi(S) \ge \sigma(S).$ We have to show the theorem true for a new multiset $S'$ formed by appending a new element $a_{n+1}\ge 2$ since any multiset can be built up one element at a time.
    
    Let $k$ be the largest integer such that $a_{n+1}>\pi(S)^k.$ From this we grab two crimps,
    $$a_{n+1}-1\ge \pi(S)^k \quad \text{and}\quad a_{n+1}<\pi(S)^{k+1},$$
    with which the last bit of the proof can be shown easily.
    $$\pi(S')=a_{n+1}\pi(S) = (a_{n+1}-1)\pi(S) + \pi(S)\ge \pi(S)^k\pi(S) + \sigma(S) $$
    $$= \pi(S)^{k+1} + \sigma(S) > a_{n+1} + \sigma(S) = \sigma(S')$$
    Technically this is a bit overkill since we've shown $\pi(S') > \sigma(S')$ when all we needed was $\pi(S') \ge \sigma(S')$. Oh well.
\end{proof}

With that Lemma we can catalog all bioperational multisets over $\mathbb{N}$ by their sum-product.

\begin{theorem}
    For every composite integer $m\in\mathbb{N}$ there exists a non-trivial bioperational multiset over $\mathbb{N}$ with a sum-product of $m$.
\end{theorem}
\begin{proof}
    Suppose a composite integer $m=a_1a_2...a_k$ with $k>1$ and $a_i\ge 2$ for $i=1,...,k$. Let $S=\{a_i\}_{i=1}^k$. By Lemma 3.3, we know $\pi(S)\ge \sigma(S)$. So let the non-negative integer $d=\pi(S) - \sigma(S)$ be their difference. The multiset
    $$S'=\{a_1,a_2,...,a_k,\overbrace{1,...,1}^{d\text{ times}}\}$$
    is bioperational with sum-product $m$ since
    $$\sigma(S') = a_1+...+a_k + \overbrace{1+...+1}^{d\text{ times}}=\sigma(S) + (\pi(S) - \sigma(S)) = \pi(S) = \pi(S')=m.$$
\end{proof}

From the proof of the previous theorem we can also make a statement about the lengths of bioperational multisets.

\begin{corollary}
    For every factorization of a composite integer $m=a_1a_2...a_k$ there exists a non-vanishing bioperational multiset over $\mathbb{N}$ of length $m + k - \sum a_i.$
\end{corollary}
\begin{proof}
    Let $S$ and $S'$ denote the same multisets as in the proof of Theorem 3.4. $S'$ is bioperational and contains $k+d = k+(\pi(S) - \sigma(S))=m+k-\sum a_i$ elements.
\end{proof}

Starting at $n=2$ the number of non-vanishing bioperational multisets over $\mathbb{N}$ of length $n$ is
$$1,1,1,3,1,2,2,2,2,3,2,4,2,...$$
(sequence A033178 in OEIS [8]). The positions of record in this list occur at 
$$n=2,5,13,25,37,41,61,85,113,181,361,421,433,...$$
(sequence A309230 in OEIS). The terms all appear to have fewer prime factors than their neighbors.

\section{Integers ($\mathbb{Z}$)}
An interesting thing happens once negatives are on the playing field. A multiset can be extended without changing either sum or product. Consider
$$S=\{1,2,3,-1,-1,1,1\}$$
which is bioperational with sum-product $6$. This is the first example of a non-minimal bioperational multiset. In $\mathbb{N}$ every non-vanishing bioperational multiset is also minimal. Not so in $\mathbb{Z}$!

Accordingly, we now use \textit{bioperation} as a verb as well. We say a multiset has been \textit{bioperated} if it has been made bioperational by means of changing its sum with appendages. For example, we may bioperate $S=\{3,-5\}$ which has a sum $\sigma(S)=-2$ and product $\pi(S)=-15$. Since appending $T=\{-1,-1,1\}$ decrements $\sigma(S)$ and fixes $\pi(S)$, bioperation is accomplished by just repeatedly appending $T$. In particular,
$$S'=S+ 13T$$
is bioperational. Note however $S'$ is not minimal. We can trim it down to minimality by shaving off groups of $\{-1,-1,1,1\}$ which have no effect on the sum-product obtaining
$$S''=S+13T-6\{-1,-1,1,1,1\}=\{3,-5,\overbrace{-1,...,-1}^{14\text{ times}},1\}$$
which is, in fact, minimal.

There are three important appendages in $\mathbb{Z}$ which fix the product.
\begin{center}
\begin{tabular}{ c|c|c } 
 label & appendage & $\Delta \sigma(S)$ \\ \hline
 $T_1$ & $\{1\}$ & $+1$ \\ 
 $T_0$ & $\{1,1,-1,-1\}$ & $0$ \\ 
 $T_{-1}$ & $\{1,-1,-1\}$ & $-1$ \\ 
\end{tabular}
\end{center}
We now give the parallel of Theorem 3.4 for $\mathbb{Z}$.
\begin{theorem}
    For every composite integer $m\in\mathbb{Z}$ there exists a non-trivial minimal bioperational multiset over $\mathbb{Z}$ with a sum-product of $m$.
\end{theorem}
\begin{proof}
    Choose a factorization $m=a_1...a_n$ with $n\ge 2$ where each $a_i$ may be positive or negative and $|a_i|\ge 2$ for $i=1,...,n$. The multiset $S=\{a_i\}_{i=1}^n$ has the desired product. Bioperate $S$ producing $S'$ such that $\sigma(S')=\pi(S')=\pi(S)$. This is done by appending $T_{\pm1}$ as needed. To be 
    
    Finally, if $S'$ is not minimal we may take a minimal bioperational multiset, $S''$, from it. $S''$ must include the non-units $a_1,...,a_n$ (a ``unit" by the way is a fancy name for a number with an inverse in its same ring; in this case $1$ and $-1$). Since $n\ge 2$ we are assured that $S''$ is non-trivial.
\end{proof}

\section{Fields}
Bioperational multisets turn out disappointingly abundant in fields.
\begin{lemma}
    Given any multiset $S=\{a_i\}_{i=1}^n$ whose elements are in a field $F$ and such that $\pi(S)\not=1$, one can bioperate $S$ into a unique multiset $S'$ by appending a single element,
    $$a_{n+1}=\frac{\sigma(S)}{\pi(S) - 1}.$$
    This was stated for $F=\mathbb{Q}$ and $n=4$ by Robert Israel in [6].
\end{lemma}
\begin{proof}
    Any element $a_{n+1}\in F$ which might bioperate $S$ must satisfy $\sigma(S) + a_{n+1}=\pi(S) a_{n+1}.$ Rearranging yields $a_{n+1}=\frac{\sigma(S)}{\pi(S) - 1}$ which exists if $\pi(S)\not=1$.
\end{proof}

The lemma turns out to be an exhaustive description.
\begin{theorem}
    In any field, all non-trivial bioperational multisets can be produced with Lemma 5.1.
\end{theorem}
\begin{proof}
    Suppose we have some bioperational multiset $S=\{a_i\}_{i=1}^n$ which cannot be produced by the lemma. Let $S_i'$ be the multiset formed by removing $a_i$. It follows from the lemma $\pi(S_i')=\frac{\pi(S)}{a_i}=1$ for all $i=1,...,n$. This in turn implies $\pi(S)=a_i$ for $i=1,...,n$ and we see that all $a_i$ are equal. We therefore have a solution to 
    $$a_1^n=na_1.$$
    But dividing out an $a_1$ from both sides gives us $n=a_1^{n-1}=\pi(S_1')=1$ showing $S$ is trivial.
\end{proof}

Before leaving this territory, we note that there are solutions to $a^{n-1}=n$ leading to bioperational sets of a single value. Take for instance $\{2,2,2,2,2\}$ which is bioperational in $\mathbb{F}_{11}.$

\section{Lunar Integers ($\mathbb{L}$)}
The Lunar Integers are the only strictly \textit{semi}-ring to be considered. Their arithmetic is well analyzed in [10] (there called ``Dismal" Arithmetic) and Neil Sloane gives a wonderful introduction in a Numberphile interview [11]. 

We neglect to explain the arithmetic here ourselves. We need only note some properties of the number of digits. If we let $D(a)$ denote the number of digits of a lunar integer $a\in\mathbb{L}$, then
$$D(ab)=D(a)+D(b)-1\quad\text{and}\quad D(a+b)=\text{max}\{D(a),D(b)\}.$$

These give us
\begin{lemma}
    In any Lunar Bioperational Set, there is at most one element with $2$ or more digits.
\end{lemma}
\begin{proof}
    Suppose $S=\{a_i\}_{i=1}^n\subset\mathbb{L}$ is bioperational and that $D(a_1)\ge D(a_i)$ for $i=2,...,n$. From the aforementioned identities
    $$D(\sigma(S))=\text{max}\{D(a_i)\}_{i=1}^n=D(a_1)$$
    and
    $$D(\pi(S))=1+\sum_{i=1}^n (D(a_i)-1)=D(a_1)+\sum_{i=2}^n (D(a_i)-1).$$
    Since $D(\pi(S))=D(\sigma(S))$, it follows that $\sum_{i=2}^n (D(a_i)-1)=0$ and hence that $D(a_i)=1$ for $i=2,...,n.$
\end{proof}

Apparently bioperational multisets can't breathe well on the moon:
\begin{theorem}
    Every minimal bioperational multiset of Lunar integers is trivial.
\end{theorem}
\begin{proof}
    We prove the contrapositive. Suppose $S=\{a_i\}_{i=1}^n\subset\mathbb{L}$ bioperational and non-trivial. From Lemma 6.1 we may assume $D(a_i)=1$ for $i=2,...,n.$ For $a\in\mathbb{L}$, let $F(a)\in\mathbb{L}$ denote be the last digit of $a$. From the definitions of addition and multiplication over $\mathbb{L}$
    $$\max\{F(a_1), a_2, ..., a_n\}=F(\sigma(S))=F(\pi(S))=\min\{F(a_i),a_2,...,a_n\}.$$
    But this implies $F(a_1)=a_2=...=a_n$. In which case the multiset $S'=\{a_1\}$ is trivially bioperational with the same sum-product as $S$ and hence $S$ is not minimal.
\end{proof}

So there are bioperational multisets in $\mathbb{L}$, like $\{17,7\}$ and $\{2,2,2\}$, but they aren't very interesting.

\section{Gaussian Integers ($\mathbb{Z}[i]$)}
Gaussian integers are numbers of the form $a+bi$ where $a$ and $b$ are integers and $i^2=-1$ (so like $2+3i$ or $-1-19i$ for example). In addition to the appendages $T_{-1},T_0,$ and $T_1$ given in Section 4, two more appear in $\mathbb{Z}[i]$,
$$T_{\pm2i}=\{\pm i, \pm i, -1, 1\},$$
which perturb the sum by $\pm 2i$ and fix the product. So sometimes bioperate the imaginary part of a multiset sum. 

Take, for instance, $S=\{1+2i, 2+3i\}$. We have $$\sigma(S)=3+5i\quad\text{and}\quad \pi(S)=-4+7i.$$
The difference is $\pi(S) - \sigma(S)=-7+2i$. We bioperate by 1) appending $T_{-1}$ seven times, 2) appending $T_{2i}$ once, and 3) shaving off $T_{0}$ until minimality is reached. The result is 
$$S'=\{1+2i, 2+3i, i, i, -1,-1,-1,-1,-1,-1,-1\}$$
which is bioperational with $\pi(S')=\sigma(S')=-4+7i.$

We need a couple lemmas before the result analogous to Theorem 3.4.
\begin{lemma}
    A Gaussian integer $a+bi$ is a multiple of $1+i$ if and only if $a$ and $b$ have the same parity (that is, are both odd or both even).
\end{lemma}
\begin{proof}
    Firstly, suppose $a+bi=(1+i)(c+di)$ is a multiple of $1+i$. Then $a=c-d$ and $b=c+d$. $a$ and $b$ therefore have the same parity since $b=a+2d$.
    
    Conversely, suppose $a$ and $b$ have the same parity. If both even, then we may write 
    $$a+bi=2\Big(\frac{a}{2}+\frac{b}{2}i\Big)=(1+i)(1-i)\Big(\frac{a}{2}+\frac{b}{2}i\Big)$$
    and are done. If both odd, then we may write
    $$a+bi=(1+i)\Big(\frac{a+b}{2}+\frac{b-a}{2}i\Big).$$
\end{proof}

\begin{lemma}
    For any Gaussian integers $\alpha_1,...,\alpha_n\in\mathbb{Z}[i]$ such that $1+i$ does not divide any $\alpha_i$,
    $$\text{Im}\Big(\prod \alpha_i\Big)\equiv\text{Im}\Big(\sum \alpha_i\Big)\mod 2.$$
\end{lemma}
\begin{proof}
    Let $\varphi(a+bi)=\overline{b\ \%\ 2}\in\mathbb{F}_2.$ It follows that $\varphi(\alpha+\beta)=\varphi(\alpha)+\varphi(\beta)$. But more interestingly, it turns out that when neither of $\alpha$ nor $\beta$ are multiples of $1+i$ we have also $\varphi(\alpha\beta)=\varphi(\alpha)+\varphi(\beta)$. From lemma 7.1 it follows that the residues of $\alpha$ and $\beta$ in $\mathbb{Z}[i]/(2)\cong\mathbb{F}_2[i]$ are in $\{1,i\}$. It is enought to check that $\varphi$ has the desired property on $\{1,i\}:$
    $$0=\varphi(1)=\varphi(1\cdot 1)=\varphi(1)+\varphi(1)=0+0=0$$
    $$1=\varphi(i)=\varphi(1\cdot i)=\varphi(1)+\varphi(i)=0+1=1$$
    $$0=\varphi(-1)=\varphi(i\cdot i)=\varphi(i)+\varphi(i)=1+1=0$$
    The lemma follows noting
    $$\overline{\text{Im}\Big(\prod \alpha_i\Big)\ \%\ 2}=\varphi\Big(\prod \alpha_i\Big)=\sum \varphi(\alpha_i)=\varphi\Big(\sum \alpha_i\Big)=\overline{\text{Im}\Big(\sum \alpha_i\Big)\ \%\ 2}.$$
\end{proof}

The enzymes of $\mathbb{Z}[i]$ have been assembled. We are ready to digest the theorem.
\begin{theorem}
    For every $\mu\in\mathbb{Z}[i]$ which factors into non-units (i.e. $\mu=\alpha\beta$ with $\alpha,\beta\not\in\{1,-1,i,-i\}$) there exists a non-trivial minimal bioperationl multiset over $\mathbb{Z}[i]$ with a sum-product of $\mu$.
\end{theorem}
\begin{proof}
    Pick some factorization $\mu=a_1...a_n$ and let $S=\{a_i\}_{i=1}^n$ with at least two $a_i$ non-units. We break into two cases.
    
    Case 1) if $\text{Im}(\pi(S))$ and $\text{Im}(\sigma(S))$ have the same parity, we may bioperate $S$ by appending $T_{\pm1}$ and $T_{\pm2i}$ as needed. The result is $S'$; bioperational with sum-product $\mu$. If $S'$ is not minimal, we may take a minimal subset $S''$. And we are assured $S''$ is non-trivial since otherwise $S''=\{\mu\}$ which implies $a_i=\mu$ for some. And $a_i=\mu$ implies all other $\alpha_j$ for $j\not=i$ are units since $\alpha_1...\alpha_{i-1}\alpha_{i+1}...\alpha_n=1$ (note we are using in this last step the fact that $\mathbb{Z}[i]$ is an integral domain).
    
    Case 2) if $\text{Im}(\pi(S))$ and $\text{Im}(\sigma(S))$ have different parities, we may suppose from Lemma 7.2 that some $\alpha_j$ is divisible by $1+i$. We create a new multiset by removing $\alpha_j$ from $S$ and appending $\{i\alpha_j,i,-1\}$. In notation,
    $$S'=\{\alpha_i\}_{i\not=j}+\{i\alpha_j,i,-1\}.$$
    The product remains unchanged since
    $$\pi(S')=\frac{\pi(S)}{\alpha_j}(i^2\alpha_j)(-1)=\pi(S).$$
    More importantly, it is claimed that $\text{Im}(\sigma(S'))$ and $\text{Im}(\sigma(S))$ have different parity. Their difference is
    $$\text{Im}(\sigma(S'))-\text{Im}(\sigma(S))=\text{Im}(\sigma(S')-\sigma(S))=\text{Im}(i\alpha_j+i-1-\alpha_j).$$
    Let $\alpha_j=a+bi$ for some integers $a$ and $b$. Substitution gives
    $$\text{Im}(\sigma(S'))-\text{Im}(\sigma(S))=\text{Im}(ai-b+i-1-a-bi)=a-b+1.$$
    From Lemma 7.1 we may suppose that $a$ and $b$ have the same parity since $1+i|\alpha_j$. It follows that $a-b+1$ is odd, that $\text{Im}(\sigma(S'))$ and $\text{Im}(\sigma(S))$ have different parity, and therefore that $\text{Im}(\sigma(S'))$ and $\text{Im}(\pi(S))=\text{Im}(\pi(S'))$ have the same parity. And so we return to the first case to bioperate $S'$.
\end{proof}

\section{Eisenstein Integers ($\mathbb{Z}[\omega]$)}
Eisenstein integers are similar to the Gaussians in that they are all of the form $a+b\omega$ where $a$ and $b$ are integers and $\omega$ is a strictly complex number such that $\omega^3=1$. Right away this gives us our first appendage,
$$T_{3\omega}=\{\omega,\omega,\omega\}.$$
One can show further show that $\omega^2=-1-\omega$ from which we get
$$T_{-2\omega}=\{-\omega,-1-\omega,-1,1,1\}.$$
Thus we have $T_\omega=T_{3\omega}+T_{-2\omega}$ and $T_{-\omega}=T_{3\omega}+2T_{-2\omega}$ at our disposal. Suprisingly, we can therefore bioperate any multiset over $\mathbb{Z}[\omega]$. Our main theorem in this section will therefore run almost identically to its analog over $\mathbb{Z}$.
\begin{theorem}
    For $\mu\in\mathbb{Z}[\omega]$ which factors into non-units there exists a non-trivial minimal bioperational multiset over $\mathbb{Z}[\omega]$ with a sum-product of $\mu$.
\end{theorem}
\begin{proof}
    Choose a factorization $\mu=\alpha_1...\alpha_n$ with 2 non-units and let $S=\{\alpha_i\}_{i=1}^n$. Bioperate $S$ with $T_{\pm1}$ and $T_{\pm\omega}$. The resulting bioperational multiset $S'$ can be shaved down to minimality without becoming trivial since the non-units cannot be trimmed off.
\end{proof}

\section{Integers $\sqrt{2}$ Appended ($\mathbb{Z}[\sqrt{2}]$)}
Lastly, we consider the integer ring of a real quadratic number field. Numbers in $\mathbb{Z}[\sqrt{2}]$ are of the form $a+b\sqrt{2}$ where $a$ and $b$ are integers (again, very similar to $\mathbb{Z}[i]$ and $\mathbb{Z}[\omega]$). We use the fact that
$$(1+\sqrt{2})(-1+\sqrt{2})=1$$
to create 
$$T_{\pm2\sqrt{2}}=\{\pm1\pm\sqrt{2}.\,\mp1\pm\sqrt{2}\}.$$

There's good reason to think that this is the best we can do (I.e. that $T_{\pm \sqrt{2}}$ doesn't exist over $\mathbb{Z}[\sqrt{2}]$). But the best proof the author could come up with for such a fact uses difficult results about quadratic number fields and complicated induction. Instead, we take a route similar to that taken through $\mathbb{Z}[i]$.
\begin{lemma}
    The number $a+b\sqrt{2}\in\mathbb{Z}[\sqrt{2}]$ is a multiple of $\sqrt{2}$ if and only if $a$ is even.
\end{lemma}
\begin{proof}
Note $(c+d\sqrt{2})\sqrt{2}=2d+c\sqrt{2}.$
\end{proof}

\begin{lemma}
    For numbers $\alpha_1,...,\alpha_n\in\mathbb{Z}[\sqrt{2}]$ let
    $$a+b\sqrt{2}=\sum\alpha_i\quad\text{and}\quad c+d\sqrt{2}=\prod\alpha_i.$$
    If no $\alpha_i$ is divisible by $\sqrt{2}$ then $b\equiv d\mod 2$.
\end{lemma}
\begin{proof}
    We again create a strange homomorphism. Let $\varphi(a+b\sqrt{2})=\overline{b\ \%\ 2}\in\mathbb{F}_2$. It follows that $\varphi(\alpha+\beta)=\varphi(\alpha)+\varphi(\beta)$. We claim if $\sqrt{2}$ divides neither $\alpha$ nor $\beta$ then $\varphi(\alpha\beta)=\varphi(\alpha)+\varphi(\beta)$ as well. From the previous lemma, we see that the residues of such $\alpha$ and $\beta$ with coefficients in $\mathbb{F}_2$ are in $\{1,1+\sqrt{2}\}$. We check $\varphi$ by hand:
    $$0=\varphi(1)=\varphi(1\cdot1)=\varphi(1)+\varphi(1)=0+0=0$$
    $$1=\varphi(1+\sqrt{2})=\varphi(1\cdot(1+\sqrt{2}))=\varphi(1)+\varphi(1+\sqrt{2})=0+1=1$$
    $$0=\varphi(1)=\varphi((1+\sqrt{2})\cdot(1+\sqrt{2}))=\varphi(1+\sqrt{2})+\varphi(1+\sqrt{2})=1+1=0$$
    We end noting
    $$\overline{d\ \%\ 2}=\varphi\Big(\prod \alpha_i\Big)=\sum \varphi(\alpha_i)=\varphi\Big(\sum \alpha_i\Big)=\overline{b\ \%\ 2}.$$
\end{proof}

It's probable that if the author knew more about ring isomorphisms, the results of this section and those of Section 7 could have been demonstrated simultaneously. 
\begin{theorem}
    For every $\mu\in\mathbb{Z}[\sqrt{2}]$ which factors into non-units there exists a non-trivial minimal bioperationl multiset over $\mathbb{Z}[\sqrt{2}]$ with a sum-product of $\mu$.
\end{theorem}
\begin{proof}
    Pick a factorization $\mu=\alpha_1...\alpha_n$ and let $S=\{\alpha_i\}_{i=1}^n,\ a+b\sqrt{2}=\sigma(S),$ and $c+d\sqrt{2}=\pi(S)$. If $b$ and $d$ have the same parity, $S$ can be bioperated into the desired result. If not, we may pick some $\alpha_j$ a multiple of $\sqrt{2}$. Letting
    $$S'=\{\alpha_i\}_{i\not=j}+\{(1+\sqrt{2})\alpha_j,-1+\sqrt{2}\}$$
    Letting $\alpha_j=x+y\sqrt{2}$, the change $\sigma(S)$ is
    $$\sigma(S')-\sigma(S)=(x+y\sqrt{2})(1+\sqrt{2})+(-1+\sqrt{2})-(x+y\sqrt{2})=(2y-1)+(x+1)\sqrt{2}.$$
    But from Lemma 9.1, we may suppose that $x$ is even and that therefore $\sigma(S')$ and $\sigma(S)$ have $\sqrt{2}$ coefficients of different parity. It follows that $\sigma(S')$ and $\pi(S')=\pi(S)$ have $\sqrt{2}$ coefficients of the same parity and that $S'$ can therefore be bioperated into the desired result.
\end{proof}

\section{Generalization and Open Problems}
Let's start this section by bundling up our main theorems into a single statement
\begin{theorem}
    If $R$ is one of $\mathbb{N},\mathbb{Z},\mathbb{Z}[i],\mathbb{Z}[\omega],$ or $\mathbb{Z}[\sqrt{2}]$ then for every $\mu\in R$ which factors into non-units, there exists a non-trivial minimal bioperational multiset over $R$ with a sum-product of $\mu$.
\end{theorem}
\begin{proof}
    Theorems 3.4, 4.1, 7.3, 8.1, 9.3.
\end{proof}

Some open problems of interest:
\begin{itemize}
    \item Does Theorem 10.1 hold over the quaternions?\\
    Order of multiplication now matters. We at least have $T_{\pm2i},T_{\pm2j},$ and $T_{\pm2k}$ at our disposal since
    $$T_{\pm2v}=(v,v,-1,1)$$
    has a product of $1$ for $v\in\{i,j,k\}$.
    \item Does Theorem 10.1 hold for all integer rings of real quadratic number fields?\\
    There are families of such rings that admit easy attack. Take for instance $\mathbb{Z}[\sqrt{t^2\pm1]}$. From
    $$(t+\sqrt{t^2\pm1})(t-\sqrt{t^2\pm1})=\mp1$$
    we can construct appendages $T_{\pm 2\sqrt{t^2\pm 1}}$ which gives us pretty good flexibility for bioperation. And in general, for $d=t(b^2t\pm 2)$ we can construct appendages $T_{\pm2b\sqrt{d}}$. The first values not covered by these parametrizations are
    $$13,19,21,22,28,29,31,33,39,41,43,44,45,46,52,53,54,55,57,58,59,61,67,69,...$$
    Perhaps $\mathbb{Z}[\sqrt{13}]$, which has a relatively large fundamental unit, is our first example for which Theorem 10.1 fails. One would think it easy to construct a counter-example ring to the theorem. However the handful of examples the author toyed with proved dead ends.
\end{itemize}

\end{document}